\documentclass[12pt,reqno,oneside]{amsart}
\usepackage{amsmath,amsthm,amsfonts,amssymb}
\usepackage{indentfirst}
\usepackage{url}
\usepackage{hyperref}
\usepackage{graphicx}
\usepackage{lineno} % para numerar las lineas. 
\theoremstyle{plain}
\newtheorem{theorem}{Theorem}[section]
\newtheorem{prop}[theorem]{Proposition}

\theoremstyle{definition}

\newtheorem{obs}[theorem]{Remark}

\numberwithin{equation}{section}
\numberwithin{figure}{section}

\usepackage[utf8x]{inputenc}
\usepackage{amsmath,amssymb,amsthm}
\usepackage[english]{babel}

\usepackage{graphicx}
\usepackage{caption}
\usepackage{subcaption}

\usepackage{algorithm}
\usepackage[noend]{algpseudocode}
\usepackage[usenames]{color}
\usepackage{lineno} 
\usepackage[foot]{amsaddr}% Para la afiliación
\newtheorem{thm}{Theorem}[section]
\newtheorem{lem}[thm]{Lemma}

\newtheorem{cor}[thm]{Corollary}

\usepackage[letterpaper]{geometry}
\begin{document}

\baselineskip=18pt

\title[Accessible Percolation]{Accessiblility Percolation with Crossing Valleys on $n$-ary Trees}

\author[Frank Duque]{Frank Duque}
\address[Frank Duque]{Instituto de Matem\'aticas UNAM and Instituto de F\'isica UASLP, Mexico}
\email{math.frankduque@gmail.com}
\thanks{Research Partially supported by FORDECYT 265667 (Mexico)}

\author[Alejandro Roldan]{Alejandro Rold\'an-Correa}
\author[Leon Alexander Valencia]{Leon A. Valencia}
\address[Alejandro Rold\'an,  Leon Valencia]{Instituto de Matem\'aticas, Universidad de Antioquia, Colombia}
\email{alejandro.roldan@udea.edu.co}
\email{lalexander.valencia@udea.edu.co}
\thanks{Research Partially supported by Universidad de Antioquia (Colombia).}

\date{\today}

%\linenumbers

\begin{abstract} 
	
In this paper we study a variation of the accessibility percolation model, this is also motivated by evolutionary biology and evolutionary computation. Consider a tree whose vertices are labeled with random numbers. We study the probability of having a monotone subsequence of a path from the root to a leaf, where any $k$ consecutive vertices in the path contain at least one vertex of the subsequence. An $n$-ary tree, with height $h$, is a tree whose vertices at distance at most $h-1$ to the root have $n$ children. For the case of $n$-ary trees, we prove that, as $h$ tends to infinity the probability of having such subsequence:  tends to 1, if $n$ grows significantly faster than $\sqrt[k]{h/(ek)}$
% ($n(h)\geq \sqrt[k]{h/(ek)}+\omega\left(\sqrt[k]{\log(h)}\right)$)
; and tends to 0, if $n$ grows significantly slower than $\sqrt[k]{h/(ek)}$
% ($n(h)\leq \sqrt[k]{h/(ek)}- \Omega(h^c) $ with $c>0$)
.

\end{abstract}

\maketitle

\section{Introduction}
\label{S: Introduction}

Let $T$ be a rooted tree (whose edges are directed from fathers to children). Assume that the vertices of $T$ are labeled with independent and identically distributed continuous random variables; given a vertex $v\in T$, we denote by $w(v)$ its label.  Let \[P=v_0\rightarrow v_1\rightarrow \cdots \rightarrow v_h\] be a path in $T$ from the root to a leaf. We say that $P$ is a $k$\textit{-accessible} if there is a subsequence  $S:=v_{r(0)},v_{r(1)},v_{r(2)},\ldots,v_{r(t)}$ of the vertices in $P$  such that: \[w({v_{r(0)}})<w({v_{r(1)}})<w({v_{r(2)}})<\cdots<w({v_{r(t)}}),\]
each $k$ consecutive vertices of $P$ contains at least one vertex in $S$, $v_0=v_{r(0)}$ and $v_h=v_{r(t)}$. In Figure~\ref{fig:kAccessiblilityInT} we illustrate a labelled tree and  its $2$-accessible paths.  We denote by $\theta_k(T)$ the probability of having a $k$-accessible path in $T$. 
\begin{figure}[htb]
	\centering
	\includegraphics[width=.6\linewidth]{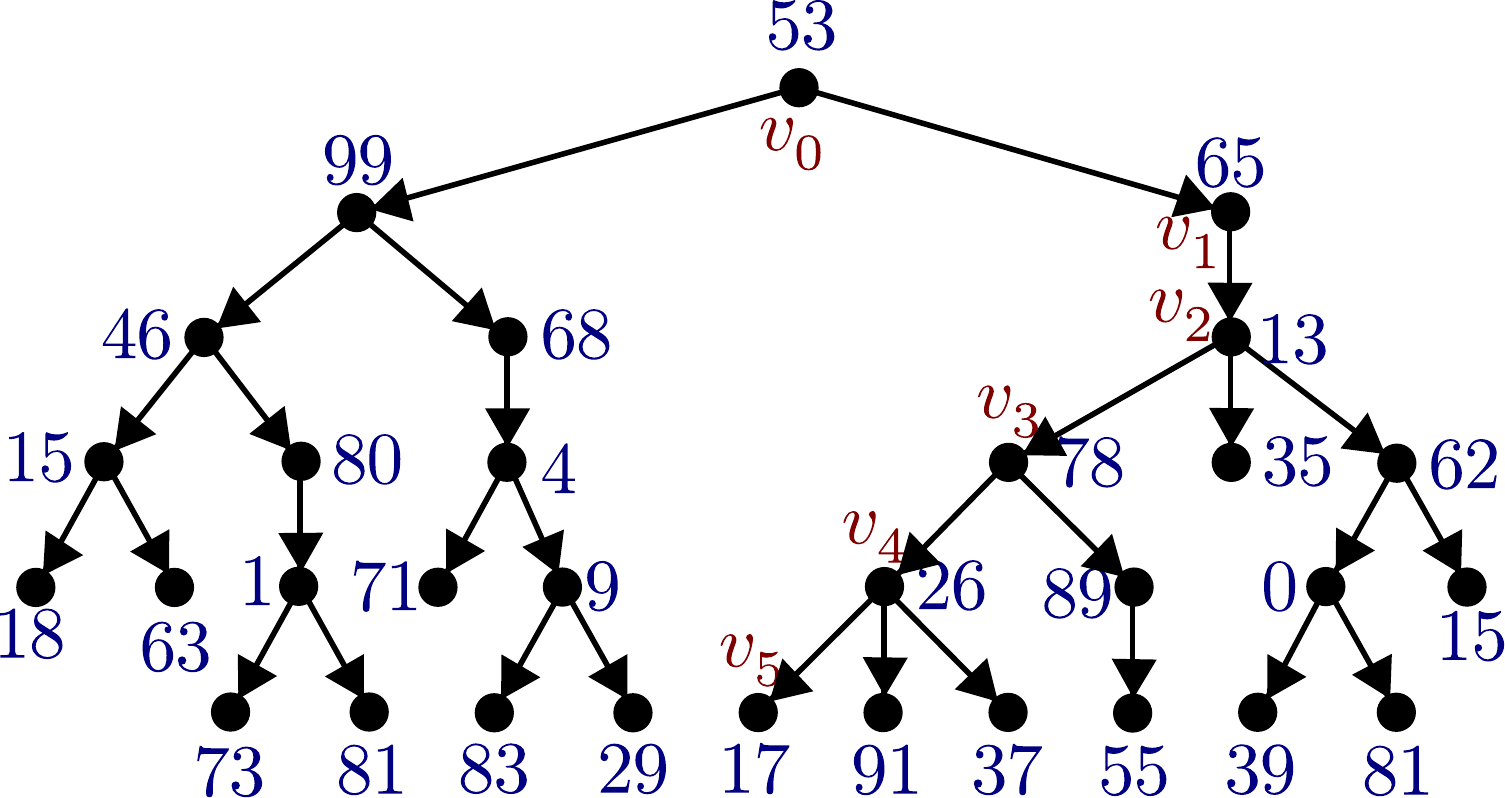}
	\caption{An example of a labelled tree without $1$-accessible paths, and for which the only $2$-accessible paths  are labeled as: $53$,$99$,$68$,$4$,$71$; and $53$,$65$,$13$,$78$,$26$,$91$.} 
	\label{fig:kAccessiblilityInT}
\end{figure}%

\subsection{Motivation.}
Fitness landscapes are by construction a static concept that assigns fitness values to the points of an underlying configuration space \cite{DefLandscapes2003}. This concept was first proposed in 1932 by Sewall Wright \cite{LandscapeDef1932} (as a mapping from a set of genotypes to fitness); since then, it has attracted particular interest in evolutionary biology and in evolutionary computation, because it offers an approach to conceptualize and visualize how an evolving population may change over time, in various population genetics models \cite{FitnessLandscape2014Book}.

For the case of an haploid asexual population on a given fitness landscape, framing by the ‘strong selection, weak mutation’ (SSWM) regime, a mutations path is considered selectively accessible if the fitness values encountered along it are monotonically increasing \cite{JKAJ2011,WDDH2006,WWC2005}. Motivated by the concept of selective accessibility, a new kind of percolation was introduced by Nowak and Krug in \cite{NK2013}, in which they studied the  probability of having a monotonically increasing path in an graph whose vertices has been labelled with random numbers \cite{NK2013,RZ2013,ApHc2016,BBS2017,Li2017}.

The model of accessibility percolation was introduced by Nowak and Krug \cite{NK2013} as follows. Imagine a population of some life form endowed with the same genetic type (genotype). If a mutation occurs, a new genotype is created which can die out or replace the old one. Provided that natural selection is sufficiently strong, the latter only happens if the new genotype has larger fitness. As a consequence, on longer timescales the genotype of the population takes a path through the space of genotypes along which the fitness is monotonically increasing \cite{G1984}.

The hypercube is a graph whose vertices are all possible $N$-tuples in $\lbrace 0,1 \rbrace^N$ (for some positive integer $N$), where two vertices are connected by an undirected edge, if the number of coordinates at which they differ is one. In many basic mathematical models of genetic mutations, the genotype sequence space is represented by the hypercube: each genome is represented as a node of the hypercube and each mutation involves the flipping of a single bit from 0 (the “wild” state) to 1 (the “mutant” state) \cite{JKAJ2011,WWC2005, Martinsson2014}. An $f$-ary (complete $f$-ary) tree with height $h$, is a rooted tree (whose edges are directed from fathers to children) whose vertices at distance at most $h-1$ to the root, have $f$ children. Being a simpler case and motivated by the hypercube case, Nowak and Krug in \cite{NK2013} study the problem of determining the growth of $f$ as a function of $h$ for having accessibility percolation in $f$-ary trees.

According to the ruggedness or smoothness of the landscape, there are some topological features (as peaks, valleys, ridges  etc.) that block the selectively accessible mutations paths towards the highest peak in the landscape. The evolutionary process that may allow escaping those topological features is known as valley crossing, which is not allowed in the selectively accessible mutation paths. However, in natural populations of sufficient size, a number of double mutants is present at all times, and the crossing valleys can be relatively facile \cite{JKAJ2011,WC2005,WDFF2009}; the SSWM assumption may therefore seem very restrictive. 

A simple way of allowing peaks and valleys on the accessibility percolation model, consists in `` allowing holes'' in the monotonicity of the path that the genotype population takes. 
In this paper, as a variation of the concept of accessibility percolation introduced by Nowak and Krug in \cite{NK2013},  we introduce the concept of $k$-accessibility percolation in which crossing small valleys is allowed. Imagine a population of some organisms  endowed with the same genotype. If a mutation occurs, a new genotype is created. The genotype die, when in the path $P$ of consecutive mutations, there is not a subsequence $S$ with increasing fitness and small holes.

\subsection{Previous results}

Given a function $f:\mathbb{Z}_+\rightarrow \mathbb{Z}_+$, we denote by $T_h(f)$ the $f(h)$-ary tree with height $h$ and we denote by $\mathcal{T}_f$ the sequence of trees $\lbrace T_h(f)\rbrace_{h\in \mathbb{Z}_+}$. We denote by $\theta_k(\mathcal{T}_f):=\lim_{h\rightarrow \infty} \theta_k (T_{h}(f))$ provided the limit exists, when $\theta_k(\mathcal{T}_f)>0$  we say that there is $k$-percolation in $\mathcal{T}_f$.

In the first work on accessibility percolation, Nowak and Krug studied the problem of determining the growth of $f$, with respect to $h$, such that there is $1$-accessibility percolation  on $\mathcal{T}_f$ \cite{NK2013}. Nowak and Krug prove that: if $f$ is a super-linear function, then there is $1$-accessibility percolation in $\mathcal{T}_f$; and if $f$ is a sub-linear function, then there is not $1$-accessibility percolation in $\mathcal{T}_f$. They also studied the linear case, $f(h)=\alpha h$, establishing the existence of a threshold between $e^{-1}$ and $1$ on the scaling constant factor; there is $1$-accessibility percolation in $\mathcal{T}_f$ when $\alpha>1$, and there is not $1$-accessibility percolation in $\mathcal{T}_f$ when $\alpha<e^{-1}$. See Theorem~\ref{thm:NK} 

\begin{thm}[Nowak-Krug~\cite{NK2013}] \label{thm:NK}
$\theta_1 (\mathcal{T}_f)>0$, if $f(h)\geq  \left\lfloor \alpha h \right\rfloor $ for $h$ large enough and $\alpha>1$. $\theta_1 (\mathcal{T}_f)=0$, if $f(h)\leq  \left\lfloor \alpha h \right\rfloor$  for $h$  large enough and $\alpha \leq e^{-1}$.
% 	\[\theta_1 (\mathcal{T}_f)
% 	\begin{cases}
% 	>0,       & \quad \text{if } f(h)\geq  \left\lfloor \alpha h \right\rfloor \text{ for $h$  large enough and } \alpha>1,\\
% 	=0,       & \quad \text{if } f(h)\leq  \left\lfloor \alpha h \right\rfloor \text{ for $h$  large enough and } \alpha \leq e^{-1}. \\
% 	\end{cases}  \]\\
\end{thm}

Continuing with the work of Nowak and Krug, Roberts and Zhao \cite{RZ2013} determined that $\alpha=e^{-1}$ is a threshold on the scaling constant factor for $1$-accessibility percolation on $\mathcal{T}_f$, see Theorem~\ref{thm:NK1p} and Theorem~\ref{thm:NK1perc}. 

\begin{thm}[Roberts-Zhao \cite{RZ2013}]\label{thm:NK1p}  
	\[\theta_1 (\mathcal{T}_f)=
	\begin{cases}
	1,       & \quad \text{if } f(h)\geq  \left\lfloor \alpha h \right\rfloor \text{ for $h$ large enough and } \alpha>e^{-1},\\
	0,       & \quad \text{if } f(h)\leq  \left\lfloor \alpha h \right\rfloor \text{ for $h$ large enough and } \alpha \leq e^{-1}.\\
	\end{cases}  \]
\end{thm} 

\begin{thm} [Roberts-Zhao \cite{RZ2013}]\label{thm:NK1perc}
If $f(h)=\left( \frac{1+\beta_h}{e} \right)h$ where $\beta_h \rightarrow 0$ as $h \rightarrow \infty$, then 
\begin{equation*}
 \theta_1 (\mathcal{T}_f)= 
	\begin{cases}
	1,      & \quad \text{if } h\beta_h/\log h \rightarrow \infty. \\
	0,      & \quad \text{if } \log h -2h\beta_h \rightarrow \infty,\\
	\end{cases}  
\end{equation*} 
\end{thm}

Accessibility percolation has been also studied,  recently, for the case when the underlying graph is the hypercube; see \cite{Martinsson2014,martinsson2015,ApHc2016,BBS2017,Li2017,BBS2017}. 

% Accessibility percolation has been also studied,  recently, for the case when the underlying graph is the hypercube; in this case it is assumed that the nodes are labelled with i.i.d. random variables, except for $(1,1,\ldots, 1)$ that is labeled by $1$ and $(0,0,\ldots, 0)$ that is labelled by a deterministic number $x$. Some results about accessibility percolation on the hypercube where obtained in \cite{Martinsson2014,martinsson2015,ApHc2016,BBS2017,Li2017,BBS2017}
% % 
% Accessibility percolation has been also studied,  recently, for the case when the underlying graph is the hypercube; in this case it is assumed that the nodes are labelled with i.i.d. random variables, except for $(1,1,\ldots, 1)$ that is labeled by $1$ and $(0,0,\ldots, 0)$ that is labelled by a deterministic number $x$.  Berestycki, Brunet and Shi studied (in \cite{ApHc2016}) the number of accessibility paths and (in \cite{BBS2017})  the existence of a threshold for accessibility percolation in the hypercube. Li Li in \cite{Li2017} complemented the work in \cite{BBS2017} by determining the threshold  for accessibility percolation in the hypercube.

\subsection{Our results.}

We denote by $\Omega(g)$ the functions $t$ such that $\lim_{h\rightarrow\infty}(t(h)/g(h))\geq k>0$.
We denote by $\omega(g)$ the functions $t$ such that $\lim_{h\rightarrow\infty}(t(h)/g(h)) = \infty$ (for a more formal definition of $\Omega$ and $\omega$ the reader may change $\lim_{h\rightarrow\infty}$ by $\liminf_{h\rightarrow\infty}$).  
Therefore $\omega\left(\sqrt[k]{\log(h)}\right)$ denotes a function $t$ such that, for any constant $C>0$, $t(h)\geq C\sqrt[k]{\log(h)}$ for $h$ large enough; similarly $\Omega(h^c)$ denotes a function that, for some constants $c,C>0$, it is above $Ch^c$ for $h$ large enough.

In this paper we  determine the growth of $f$, as a function of $h$, for which there is $k$-accessibility percolation on $\mathcal{T}_f$. Our main result is the following.

\begin{thm}\label{thm:main}
Let $\mathcal{T}_f$ be the sequence of $n$-ary trees $\lbrace T_h(f)\rbrace_{h\in \mathbb{Z}_+}$. Then
\begin{equation} \label{eq:mainthm}
 \theta_k(\mathcal{T}_f)=
	\begin{cases}
	1,       & \quad \text{if } f(h)\geq  \sqrt[k]{h/(ek)}+\omega\left(\sqrt[k]{\log(h)}\right), \\
	0,       & \quad \text{if } f(h)\leq  \sqrt[k]{h/(ek)}- \Omega(h^c) \text{ and } c>0.\\
	\end{cases} 
\end{equation}
\end{thm} 

The reader may note that, from Theorem~\ref{thm:main} it follows that 

\begin{cor}
\[\theta_k(\mathcal{T}_f)=
	\begin{cases}
	1,       & \quad \text{if } f(h)\geq \sqrt[k]{h/(ek)}c \text{ for $h$ large enough and } c>1, \\
	0,       & \quad \text{if } f(h)\leq \sqrt[k]{h/(ek)}c \text{ for $h$ is large enough and } c<1.
	\end{cases}  \]
\end{cor} 

% Instead of analyzing the first and second moment of the number of accessible paths,  in the proof of Theorem \ref{thm:main},  we use combinatoric techniques and Theorem \ref{thm:NK1perc}. There are two main reasons for this. In the first case, is easer to use combinatoric techniques. In the second case, we obtain a better bound, see Remark \ref{remark}.   
% 
% \begin{obs}\label{remark}
% Consider $P$ as a path from the root to a leaf in $T_h(f)$. 
% %Let $A_P$ be the event $\{ P  \text{ is $k$-accessible} \}$. 
% Using that $\mathbf{P}(\ P  \text{ is $k$-accessible}) \leq k^{h/k}/ \lfloor (h/k)\rfloor!$, we obtain  \[ \theta_k(T_h(f)) \leq f(h)^h \frac{k^{h/k}}{\lfloor h/k \rfloor!} \sim \left( f(h)\sqrt[k]{\frac{k^2 e}{n}}  \right) ^n \sqrt{2\pi n/k}.  \]
% Therefore $\theta_k(\mathcal{T}_f)=0$ whether $f(h)\leq  \left\lfloor\sqrt[k]{h/(ek^2)}c\right\rfloor$. This bound does not attain the value determined by Theorem \ref{thm:main}.
% \end{obs}

\section{Preliminaries}

Before proceeding with the proof of Theorem~\ref{thm:main}, we introduce the concept of $k$-transitive closure any we state an equivalent version of Theorem~\ref{thm:NK1perc}. 

We define the $k$-transitive closure of $G$, $G^k$, as the graph  obtained from it, by adding new edges from each vertex $u$ to each vertex $v$, with the property that $G$ does not already contain the directed edge from $u$ to $v$ but does contain a directed
path from $u$ to $v$  with length at most $k$. 

Although the concept of $k$-accessibility percolation was only introduced for trees, it can be easily extended  to directed graphs where: there is a single vertex distinguished as the source, some vertices distinguished as sinks, and all of the maximal paths are directed paths from the source to some sink; we name the graphs that satisfy the previous statements as monotone graphs. For the case of monotone graphs, we say that a path in it is $1$-accessible if: it starts in the source, ends at some sink and its vertices have increasing labels. Similar to the case of trees, we denote by $\theta_1(G)$ the probability of having a $1$-accessible path in $G$.

In Figure~\ref{fig:1AccessiblilityInT3} we illustrate the $2$-transitive closure of the graph depicted in Figure~\ref{fig:kAccessiblilityInT} and its $1$-accessible paths. Note that the $1$-accessible paths in the graph depicted in Figure~\ref{fig:1AccessiblilityInT3} correspond to the $2$-accessible paths in the graph depicted in Figure~\ref{fig:kAccessiblilityInT}. 

\begin{obs}\label{def1}
	 Let $T$ be a rooted tree with height $h$ and $T^k$ be its $k$-transitive closure. Then $\theta_1(T^k)=\theta_k(T)$.
\end{obs}
\begin{proof}
 Let $T$ be a rooted tree with height $h$ and $T^k$ be its $k$-transitive closure. 
 If $P$ is a $k$-accessible path in $T$, there is a subsequence $S$ of the vertices in $P$ with increasing labels, that contains at least one vertex in each $k$ consecutive vertices in $P$ and that contains the root and a leaf of $T$; therefore $S$ is a $1$-accessible path in $T^k$.
 On the other direction, 
 if there is a $1$-accessible path in $T^k$, the vertices in such path define a subsequence $S$ of the vertices in some path $P$ in $T$ that make $P$ a $k$-accessible path.
\end{proof}

\begin{figure}[htb]
	\centering
	\includegraphics[width=.6\linewidth]{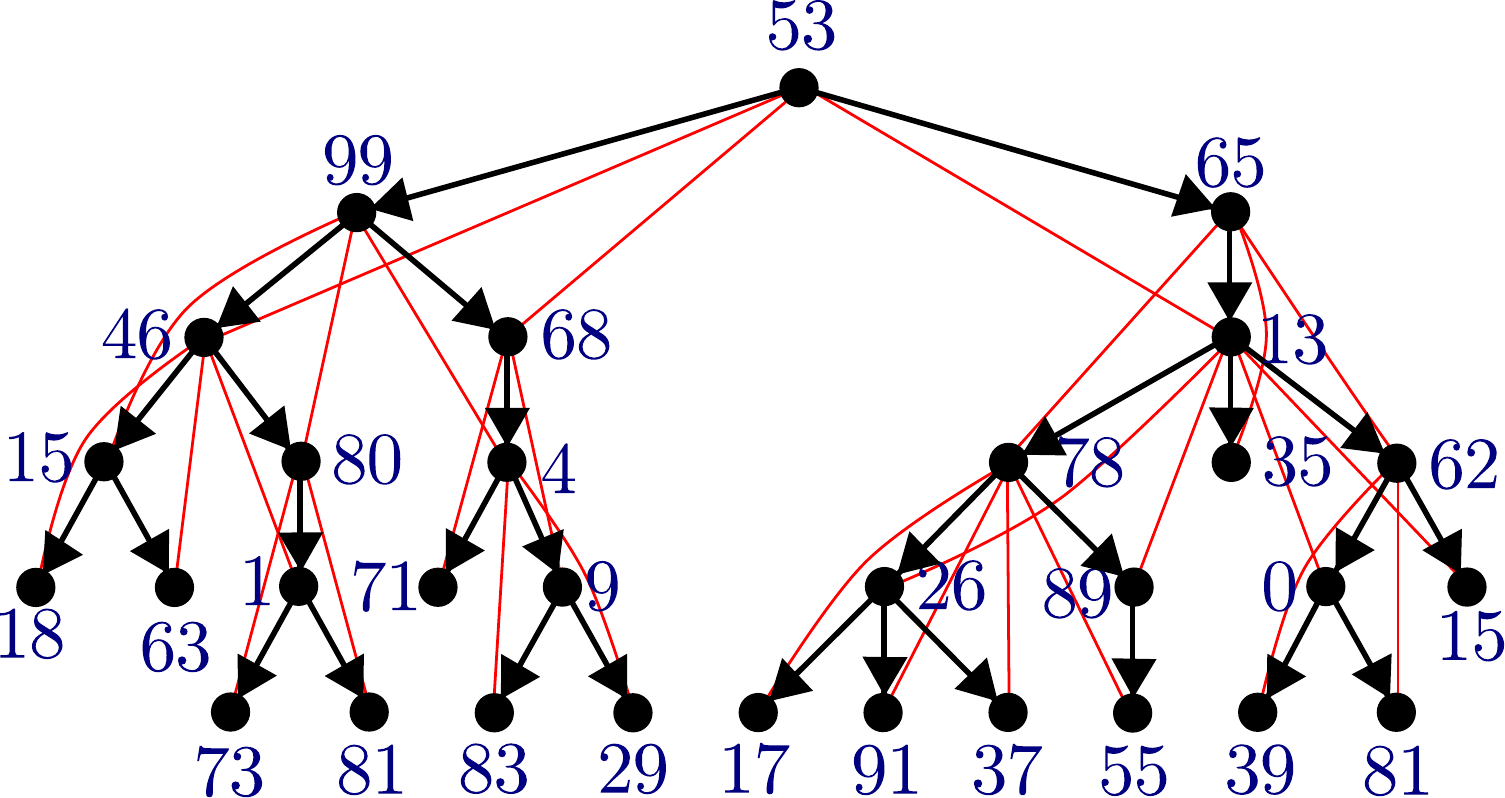}
	\caption{An illustration of the $2$-transitive closure of the graph depicted in Figure~\ref{fig:kAccessiblilityInT}, the $1$-accessible paths in such graph are labeled as: $53$,$68$,$71$; and $53$,$65$,$78$,$91$.} 
	\label{fig:1AccessiblilityInT3}
\end{figure}%

For proving Theorem~\ref{thm:main} we use the following proposition, that it is an equivalent version of Theorem~\ref{thm:NK1perc}.

\begin{prop}\label{prop:nk}
\begin{equation*}
 \theta_1 (\mathcal{T}_f) = 
	\begin{cases}
	1,      & \quad \text{if } f(h)= \frac{h}{e}+ \omega\left(\log(h) \right),  \\
	0,      & \quad \text{if } f(h)= \frac{h}{e}+ \frac{\log(h)}{2e} - \omega(1).\\
	\end{cases}  
\end{equation*} 
\end{prop}

\begin{proof}
 Let $f(h)$, $g(h)$ and $\beta_h$ be such that
 \[f(h)= \frac{h}{e}+ g(h)\log(h)= \frac{1+\beta_h}{e}h.\] 
 As $\beta_h=g(h)\frac{h}{\log(h)}$, it follows that, $h\beta_h/\log h$ goes to infinity when  $h$ goes to infinity, if and only if $g(h)=\omega(1)$. Therefore, by Theorem~\ref{thm:NK1perc}, if $f(h)= \frac{h}{e}+\omega(\log(h))$ then  $f(h)= \frac{h}{e}+\omega(1)\log(h)$ and $\theta_1 (\mathcal{T}_f)=1$.

 Let $f(h)$, $g(h)$ and $\beta_h$ be such that
 \[f(h)= \frac{h}{e}+\frac{\log(h)}{2e}- g(h)= \frac{1+\beta_h}{e}h.\] 
 As $\beta_h=\frac{\log(h)- 2e\cdot g(h)}{2h}$, it follows that, $\log h -2h\beta_h$ goes to infinity when  $h$ goes to infinity, if and only if $g(h)=\omega(1)$. Therefore, by Theorem~\ref{thm:NK1perc}, if $f(h)= \frac{h}{e}+\frac{\log(h)}{2e}-\omega(1)$ then $\theta_1 (\mathcal{T}_f)=0$.
  
\end{proof}

\section{Proof of Theorem~\ref{thm:main} }

Now we proceed to prove Theorem~\ref{thm:main}. The proof is divided two steps according with the two cases in Equation~\ref{eq:mainthm}.

\subsection*{Step 1}
Let $\mathcal{T}_f$ be the sequence of $n$-ary trees $\lbrace T_h(f)\rbrace_{h\in \mathbb{Z}_+}$. In this section we assume $$f(h)\geq  \sqrt[k]{h/(ek)}+\omega \left( \sqrt[k]{\log(h)} \right)$$ and we prove that $\theta_k(\mathcal{T}_f)=1$.

Let $\mathcal{T}^k_f$ be the sequence of graphs $\lbrace T_h^k(f)\rbrace_{h\in \mathbb{Z}_+}$ where $T_h^k(f)$ is the $k$-transitive closure of $T_h(f)$. 
Let $g(h)= h/e + \omega \left(\log(h) \right)$, $T^k=T_h^k(f)$ and $T'=T_{\left\lfloor h/k\right\rfloor }(g)$. By Remark~\ref{def1}, provided the limit exists,
\[ \theta_k(\mathcal{T}_f)=\lim_{h\rightarrow \infty}\theta_k(T_h(f))=
\lim_{h\rightarrow\infty}\theta_1(T_h^k(f)).\] 
By Proposition~\ref{prop:nk} 
\[\lim_{h\rightarrow\infty}\theta_1(T_h(g))=\theta_1( \mathcal{T}_g)=1.\]
To prove that $\theta_k(\mathcal{T}_f)=1$, it is enough to show that 
$\theta_1(T^k)\geq \theta_1(T')$.

Let $G$ be the subgraph of $T^k$, obtained from removing the vertices in $T$, whose distance from the root is not a multiple of $k$. Note that $G$ is a tree contained in $T^k$ with height $\left\lfloor h/k\right\rfloor$. 
Also note that, for $h$ large enough, 
\[ \text{deg}_G(v)\geq \left( \sqrt[k]{h/(ek)}+\omega \left(\sqrt[k]{\log(h)} \right) \right) ^k \geq {h/(ek)}+\omega \left({\log(h)} \right) \geq  g(\left\lfloor h/k\right\rfloor);\]
thus the non leaves vertices of $G$ have degree at least $g(\left\lfloor h/k\right\rfloor)$. Therefore, for $h$ large enough $T'\subset G \subset T^k.$  Therefore, the probability of having at least one  $1$-accessible path in $T'$, is a lower bound of the probability of having at least an $1$-accessible path in $T^k$. 

\subsection*{Step 2}
Let $\mathcal{T}_f$ be the sequence of $n$-ary trees $\lbrace T_h(f)\rbrace_{h\in \mathbb{Z}_+}$. In this section we assume $f(h)\leq  \sqrt[k]{h/(ek)}- \Omega(h^c)$ for some $0<c<1$, and we prove that $\theta_k(\mathcal{T}_f)=0$.

Let $\mathcal{T}^k_f$ be the sequence of graphs $\lbrace T_h^k(f)\rbrace_{h\in \mathbb{Z}_+}$ where $T_h^k(f)$ is the $k$-transitive closure of $T_h(f)$. Let $g(h)=\frac{h}{e}$, $T=T_h(f)$, $T^k=T_h^k(f)$ and $T'=T_{ h/k}(g)$. By Remark~\ref{def1}, provided the limit exists,
\[ \theta_k(\mathcal{T}_f)=\lim_{h\rightarrow \infty}\theta_k(T_h(f))=
\lim_{h\rightarrow\infty}\theta_1(T_h^k(f)).\] 
By Proposition~\ref{prop:nk}
\[\lim_{h\rightarrow\infty}\theta_1(T_h(g))=\theta_1( \mathcal{T}_g)=0.\]
To prove that $\theta_k(\mathcal{T}_f)=0$, it is enough to show that 
$\theta_1(T^k)\leq \theta_1(T')$. For this we define a tree $H^k$ such that
\[ \theta_1(T^k)\leq \theta_1(H^k) \leq  \theta_1(T').\]

$H^k$ will denote a tree, whose vertices are in correspondence with the paths in $T^k$ that start in the root.  As an outline of the construction of $H^k$ note that, any  path in $T^k$ is characterized by the following two issues. The first issue is the vertex $v$ in which the path ends. Suppose that $v$ is at distance $l$ to the root. The second issue is the information about the levels in which it has no vertices; it can be represented by a subset $s$ of $\lbrace 1,2,\ldots,l-1 \rbrace$ that does not contain $k$ consecutive numbers. Therefore the vertices in $H^k$ will be determined by a vertex $v$ in $T^k$ and a $s$ subset of $\lbrace 1,2,\ldots,l-1 \rbrace$ that does not contain $k$ consecutive numbers. About the edges in $H^k$, two vertices will be adjacent in $H^k$ if: one of the corresponding paths in $T^k$ contains the other corresponding path, and those corresponding only differ in one edge. 
As an example of this outline, consider the paths $v_0,v_1,v_3$ and $v_0,v_1,v_3,v_5$ in the graph $T^2$ illustrated in Figure~\ref{fig:ex-H:Tg2}; in $H^2$, illustrated in Figure~\ref{fig:ex-H:H}, those paths are represented by $v_3^{ \{2\}}$ and $v_5^{ \{2,4\}}$, respectively, and they are adjacent. 

We define $H^k$ formally as follows.

Let $\mathcal{A}_l$ be the set whose elements are the subsets $s$ of $ \lbrace 1,2,\ldots ,l-1 \rbrace$ that do not contain $k$ consecutive numbers. Let $H^k=(V,E)$ be the graph whose vertices and edges are defined as follows. See Figure~\ref{fig:ex-H}.
\begin{figure} [htb]
	\begin{subfigure}{.2\textwidth}
		\centering
		\includegraphics[width=.26\linewidth]{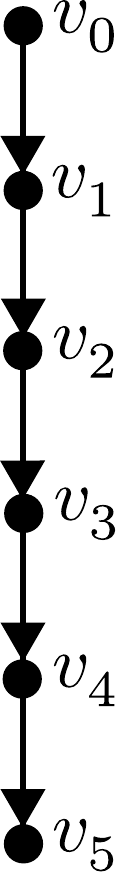}
		\caption{}
		\label{fig:ex-H:Tg}
	\end{subfigure}%
	\begin{subfigure}{.2\textwidth}
		\centering
		\includegraphics[width=.5\linewidth]{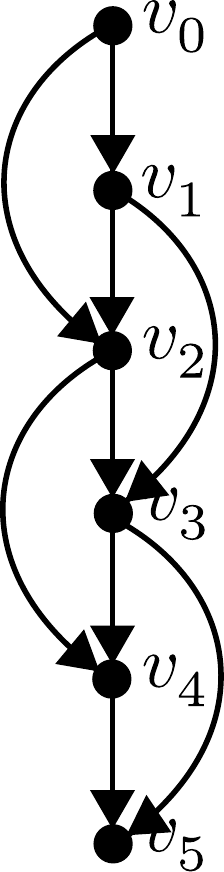}
		\caption{}
		\label{fig:ex-H:Tg2}
	\end{subfigure}%
	\begin{subfigure}{.6\textwidth}
		\centering
		\includegraphics[width=0.9\linewidth]{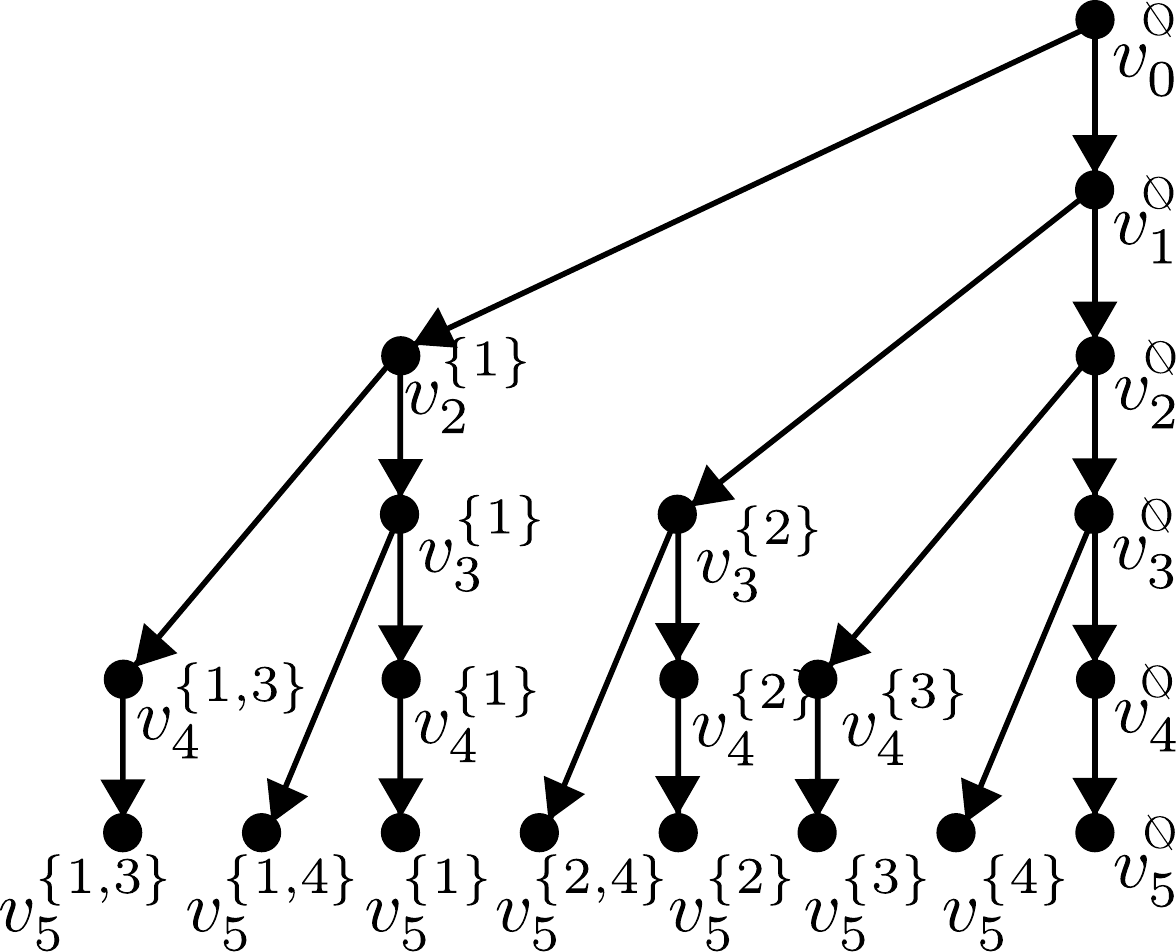}
		\caption{}
		\label{fig:ex-H:H}
	\end{subfigure}%
	\caption{An example of the subgraphs of $T$ (\ref{fig:ex-H:Tg}), $T^2$ (\ref{fig:ex-H:Tg2}) and $H^2$ (\ref{fig:ex-H:H}), respectively, corresponding to six consecutive vertices in the path $v_0, v_1, v_2, v_3, v_4, v_5, v_5, v_5$ depicted Figure~\ref{fig:kAccessiblilityInT}. }
	\label{fig:ex-H}
\end{figure}
The set of vertices in $H^k$ is defined as 
\[V:=\lbrace v^s:v\in T \textnormal{ and } s\in \mathcal{A}_l \textnormal{, where $l$ is the distance of $v$ to the root} \rbrace. \]
The edges in $H^k$ are defined recursively as follows. 
Let $v$ be a vertex of $T$ at distance $l$ to the root and let $s\in \mathcal{A}_l$. 
We say that $w^{s'}$ is a son of $v^s$ if:
\begin{itemize}
	\item $w$ is a son of $v$ in $T$ and $s=s'$
	\item $w$ is a descendant of $v$ at distance $1<j\leq k$ in $T$ and $s'=s\cup\lbrace l+1,l+2,\ldots l+j-1 \rbrace$
\end{itemize}

Given a vertex $v$ in $T^k$, we claim that: the paths in $T^k$ from the root to $v$, the elements in $\mathcal{A}_l$, and the paths in $H^k$ from the root to $v^s$ (for some $s\in\mathcal{A}_l$),  are in one-to-one correspondence. For each path $P$ in $T^k$ from the root to $v$, there exist one and only one $s$ in $\mathcal{A}_l$ whose elements are the indices $i$ such that $v_i$ is not in $P$. For each $s$ in $\mathcal{A}_l$, there is one and only one path in $P$ that starts in the root and ends in $v$; such path contains the vertices $v_i$ such that $i$ is not in $s$. For each path in $H^k$ from the root to $v^s$, obviously there is one and only one $s\in \mathcal{A}_l$. Now suppose that $s$ is in $\mathcal{A}_l$, there is one and only one path $P$ in $H^k$ from the root to $v^s$; such path contains the vertices $v_i^{s(i)}$ where $i$ is not in $s$ and $s(i)=s\cap \{1,2,\dots, i-1\}$. As an example consider: In the graph $G^2$ for the graph $G$ in Figure~\ref{fig:kAccessiblilityInT}, the path $v_0,v_1,v_3,v_5$; in $\mathcal{A}_5$, the set $\{2,4\}$; and in $H^2$, the path $v_0^{\phi},v_1^{\phi},v_3^{\{2\}},v_5^{\{2,4\}}$ (See Figure~\ref{fig:ex-H}). 

Assume that the vertices in $H^2$ are labelled with independent and identically distributed random variables with the same distribution as the labels in $T$.

 \begin{figure} [htb]
  \begin{subfigure}{.3\textwidth}
  \centering
  \includegraphics[width=.34\linewidth]{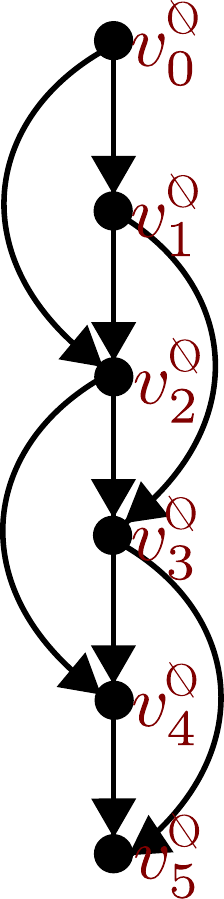}
  \caption{}
  \label{fig:ex-H:H0}
 \end{subfigure}%
 \begin{subfigure}{.33\textwidth}
  \centering
  \includegraphics[width=.36\linewidth]{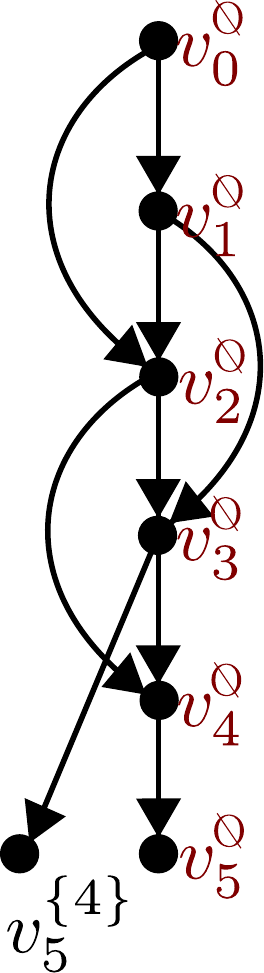}
  \caption{}
  \label{fig:ex-H:H1}
  \end{subfigure}%
  \begin{subfigure}{.36\textwidth}
  \centering
  \includegraphics[width=0.50\linewidth]{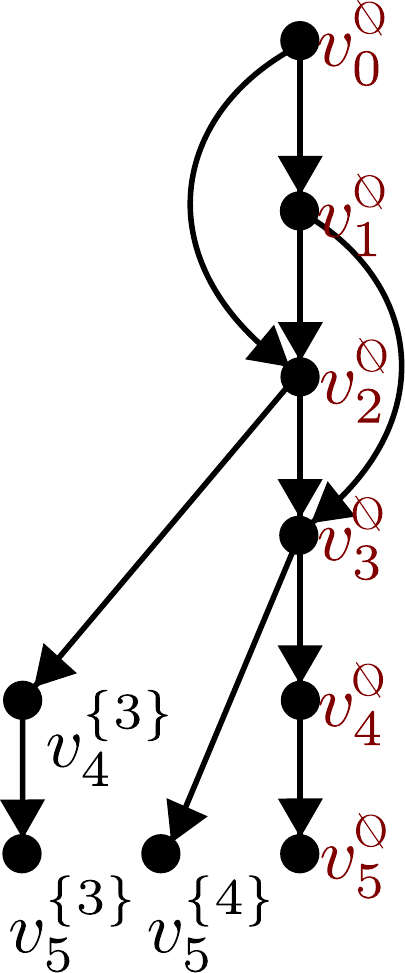}
  \caption{}
  \label{fig:ex-H:H2}
 \end{subfigure}%

 \begin{subfigure}{.4\textwidth}
  \centering
  \includegraphics[width=.73\linewidth]{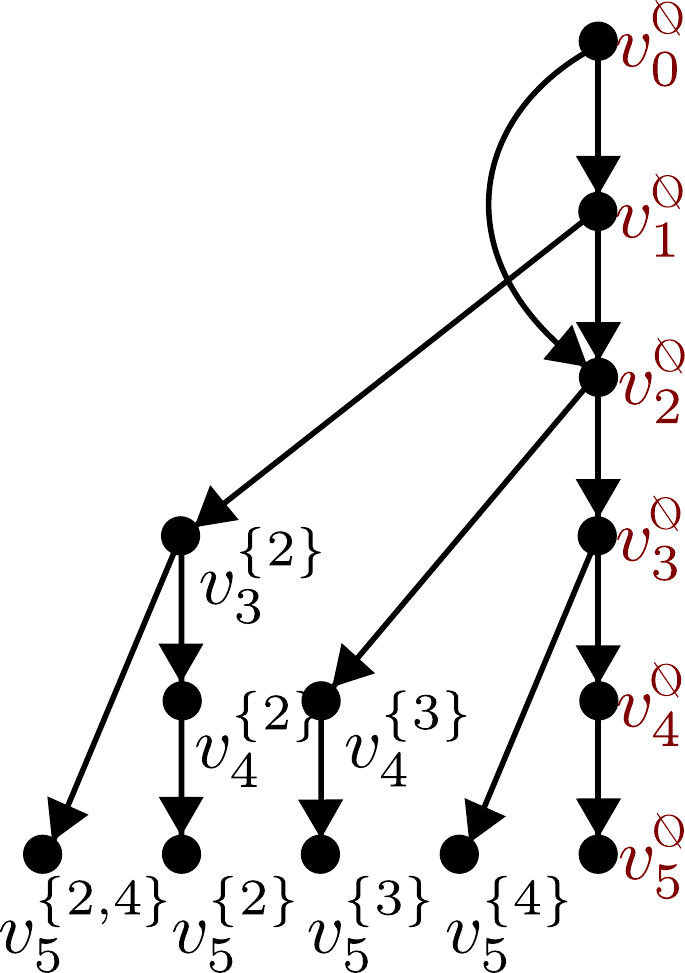}
  \caption{}
  \label{fig:ex-H:H3}
 \end{subfigure}%
 \begin{subfigure}{.6\textwidth}
  \centering
  \includegraphics[width=.85\linewidth]{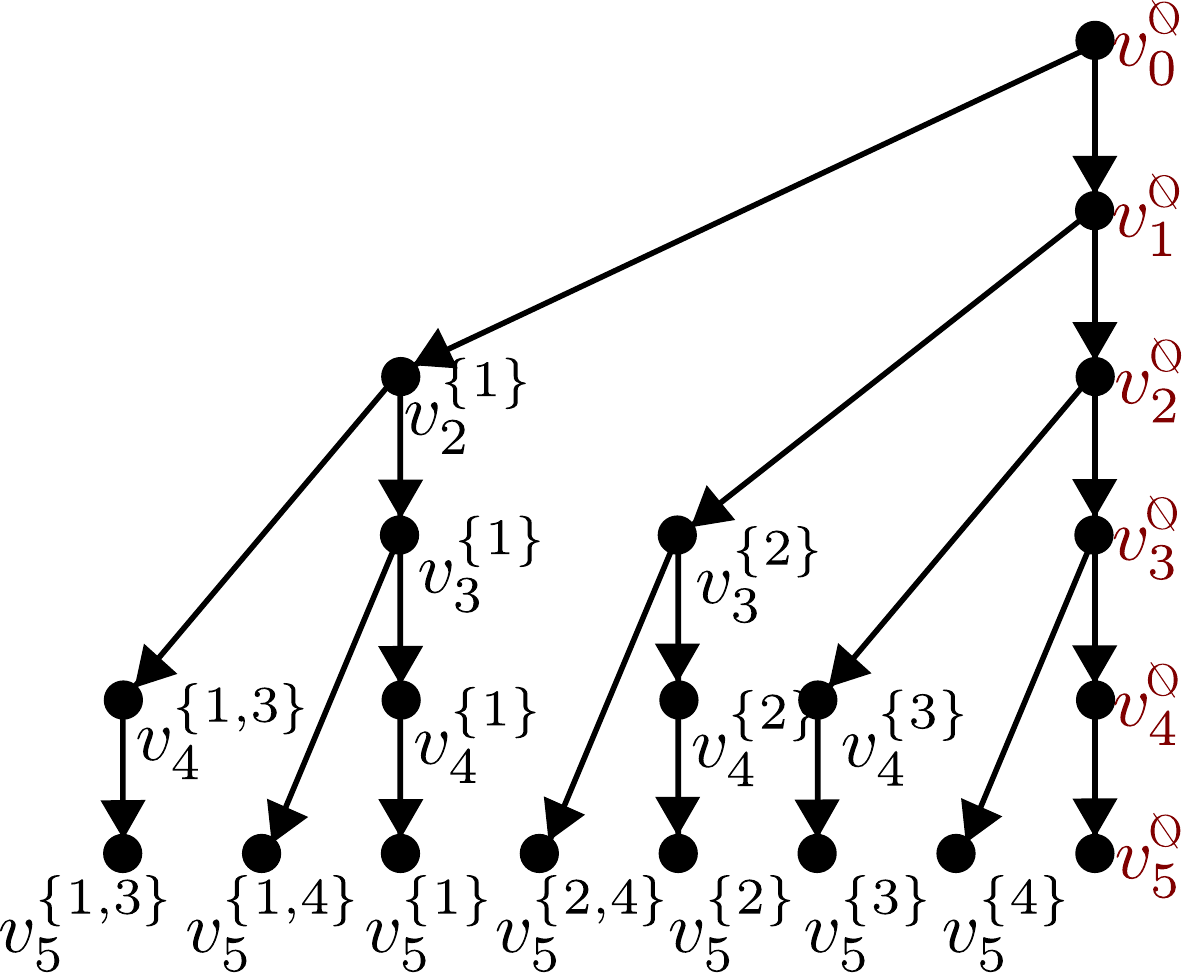}
  \caption{}
  \label{fig:ex-H:H4}
 \end{subfigure}%
 \caption{An example of a sequence of graphs in which the graph depicted in Figure~\ref{fig:ex-H:H} is obtained from the graph in Figure~\ref{fig:ex-H:Tg2}. It illustrates how to obtain the sequence of graphs used in Lemma~\ref{lem:Tg2-H-2}.}
 \label{fig:ex-H-etapas}
 \end{figure}

\begin{lem}\label{lem:Tg2-H-2}
	$\theta_1(H^k)\geq \theta_1(T^k).$
\end{lem}
\begin{proof}
 In what follows we define a sequence of graphs $H_0, H_1, \ldots, H_t$, such that $H_0=T^k$, $H_t=H^k$ and $\theta_1(H_i)\leq \theta_1(H_{i+1})$, for $i=0,1,\ldots, t-1$.
 The graph $H_0$ is obtained from $T^k$ changing each vertex $v$ for a vertex $v^\phi$, see Figure~\ref{fig:ex-H:H0}. Given some $H_i$, we say that a vertex $v\in H_i$ is divisible if: (see vertex $v_j^s$ in Figure~\ref{fig:divisiones1})
 \begin{itemize}
  \item The subgraph of $H_i$, induced by $v$ and its descendants vertices, is a directed tree. We denote such tree as $T(v)$.
  \item There are at least two edges that starts in an ancestor of $v$ and ends at $v$. From such edges we denote by $e(v)$ the edge that starts in the ancestor $x=u^s$ of $v$ such that $u$ is the nearest to the root in $T$.
 \end{itemize}
	\begin{figure} [htb]
		\begin{subfigure}{.5\textwidth}
			\centering
			\includegraphics[width=.9\linewidth]{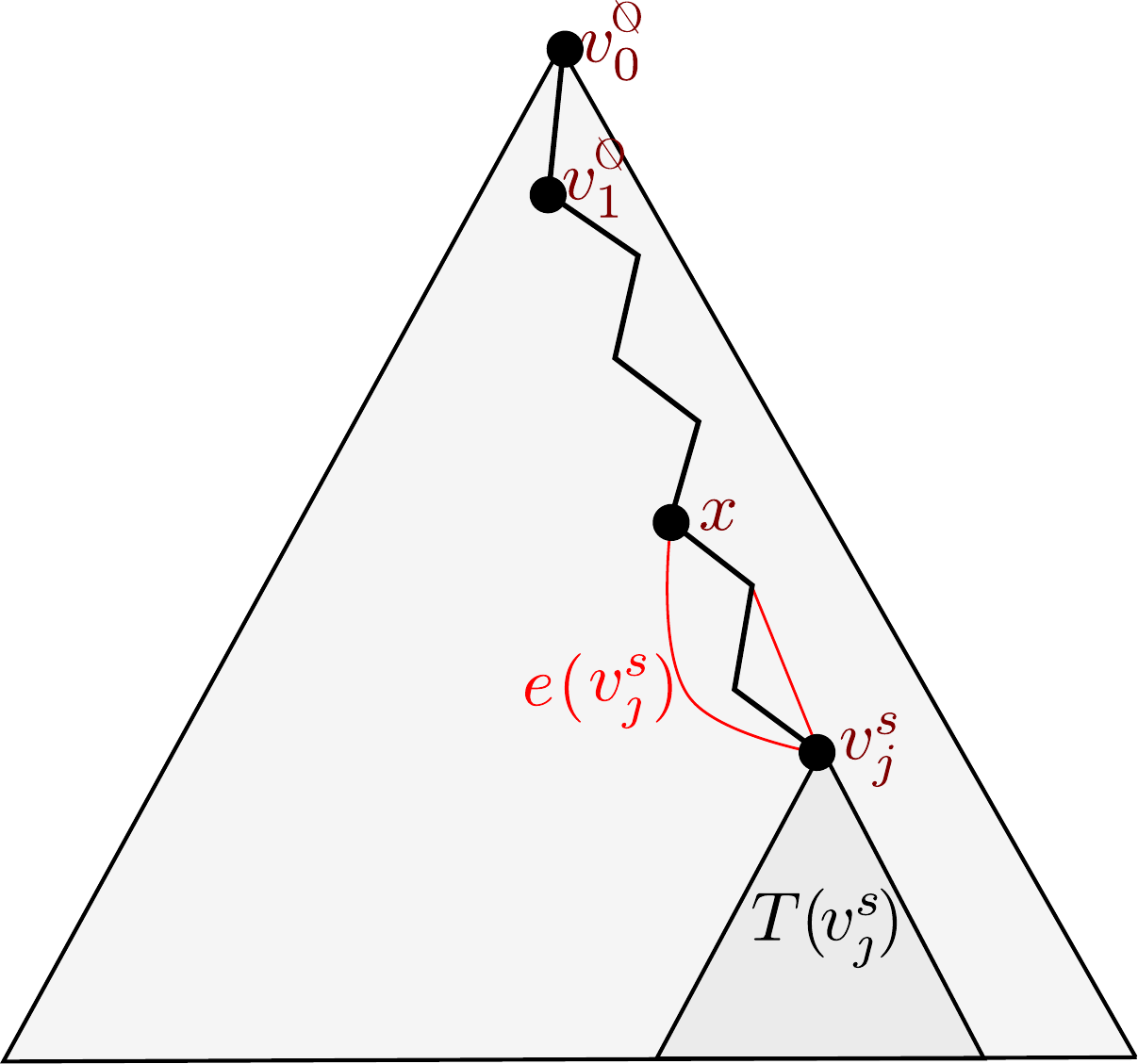}
			\caption{}
			\label{fig:divisiones1}
		\end{subfigure}%
		\begin{subfigure}{.5\textwidth}
			\centering
			\includegraphics[width=.9\linewidth]{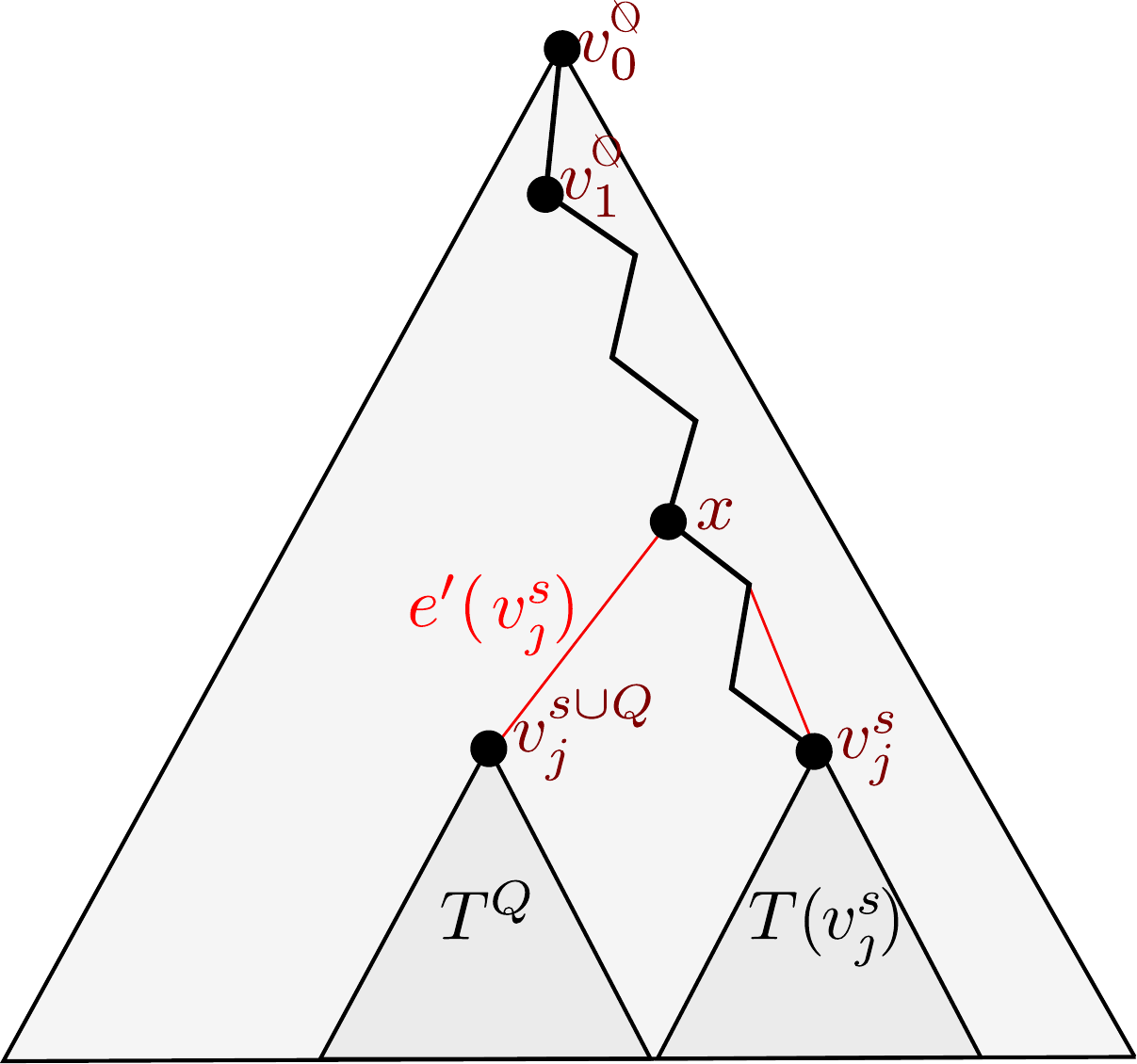}
			\caption{}
			\label{fig:divisiones2}
		\end{subfigure}%
		\caption{ Illustration of the proof of $\theta_1(H_i)\leq \theta_1(H_{i+1})$. }
		\label{fig:divisiones}
	\end{figure}
Whenever $H_i$ has at least one divisible vertex, we define $H_{i+1}$ as follows. Let $v^s$ be a divisible vertex in $H_i$. Let $v_0,v_1, \ldots, v_j$ be a path in $T$ from the root to $v$, then $v_0^{s(0)},v_1^{s(1)}, \ldots, v_j^{s(j)}$ is a path in $H_i$ from the root to $v^s$. Let $x=v_{j'}^{s(j')}$ be the vertex where $e(v)$ starts. Let $Q$ be the indices of the vertices that $e(v)$ jumps, $i.e.$ $Q:=\lbrace \{ j'+1, j'+2, \ldots j-1\} \rbrace$. Let $T^Q$ be a copy of $T(v)$ where each vertex $u^{s'}$ is changed by $u^{s'\cup Q }$. $H_{i+1}$ is defined as the graph obtained from $H_i$ by: removing $e(v)$, adding $T^Q$, and adding an edge $e'(v)$ that starts in $x$ and ends at the root of $T^Q$. See Figure~\ref{fig:ex-H-etapas}.

 The reader may notice that, if $H_i$ does not have divisible vertices then $H_i=H^k$. It remains to prove that $\theta_1(H_i)\leq \theta_1(H_{i+1})$.
 
 Given a vertex $y$, in a labeled directed graph $G$, we denote by $[y\leadsto G]$ the event of having a path in $G$ that starts in $y$,  ends at a sink of $G$ and has increasing labels; we denote by $[y\not\leadsto G]$ its complement. For the case when $y$ is the root it is denoted by $0$. 
 
 Consider in the graph in Figure~\ref{fig:divisiones1}, the event of having an $1$-accessible path  that contains the edge $e(v_j^s)$.  Note that such event occurs, if and only if, there is an increasing path from $0$ to $x$, $w(x)<w(v^s_j)$ and $[v^s_j \leadsto T(v_j^s)]$. Similarly, consider the event in Figure~\ref{fig:divisiones2}, the event of having an $1$-accessible path  that contains the edge $e'(v_j^s)$.  Note that such event occurs, if and only if, there is an increasing path from $0$ to $x$, $w(x)<w(v_j^{s\cup Q})$ and $[v_j^{s\cup Q} \leadsto T^Q]$.
 With the notation introduced in the construction of $H_{i+1}$, define $\Gamma_1$ has the event of having an $1$-accessible path in $H_i$ that contains the edge $e(v_j^s)$, and define $\Gamma_2$ has the event of having an $1$-accessible path in $H_{i+1}$ that contains the edge $e'(v_j^s)$. 
 
 This lemma follows from the following facts:
 \begin{itemize}
  \item Note that $\Gamma_1$ is the event of having an increasing path from $0$ to $x$, $w(x)<w(v_j^s)$ and $[v_j^s \leadsto T(v_j^s)]$; and $\Gamma_2$ is the event of having an increasing path from $0$ to $x$, $w(x)<w(v^{s\cup Q}_j)$ and $[v^{s\cup Q}_j \leadsto T^Q]$. Therefore, as
 \begin{align*}
  \mathbf{P}\left( w(x)<w(v^{s}_j) \right)&=\mathbf{P}\left( w(x)<w(v^{s\cup Q}_j) \right) \text{ and }\\
  \mathbf{P}\left( [v_j^s \leadsto T(v_j^s)] \right)&=\mathbf{P}\left( [v^{s\cup Q}_j \leadsto T^Q] \right) \text{ then }\\
  \mathbf{P}\left( \Gamma_1 \right)&=\mathbf{P}\left( \Gamma_2 \right).
 \end{align*}
  \item   As the only edge that joins $T^Q$ with the vertices in $H_i$ is $e'(v_j^s)$, and $H_i$ and $H_{i+1}$ only differ in $e(v_j^s)$, $e'(v_j^s)$ and $T^Q$ then
  $H_{i} -e(v_j^s)=H_{i+1} -\lbrace e'(v_j^s),T^Q\rbrace$.
  Therefore, as there are no paths in $H_{i+1} -e'(v_j^s)$ from the root to a leaf in $T^Q$ then  
  \begin{align*}
   \mathbf{P} \left(  [0\leadsto H_{i+1} -e'(v_j^s)] \right) 
   &= \mathbf{P} \left(  \left[0\leadsto H_{i+1} -\lbrace e'(v_j^s),T^Q\rbrace \right] \right)\\
   &=\mathbf{P} \left( [0\leadsto H_{i} -e(v_j^s)]  \right)
  \end{align*}
 
 \item Let $[0\leadsto x]$ the event of having an increasing path from $0$ to $x$. Then 
  \begin{align*}
   \mathbf{P}\left([0\leadsto H_i-e(v)] \big| \Gamma_1\right)
   &\geq \mathbf{P}\left([0\leadsto H_i-e(v)] \big| [0\leadsto x],w(x)<w(v^s_j) \right)\\
   & = \mathbf{P}\left([0\leadsto H_{i+1} -\lbrace e'(v_j^s),T^Q\rbrace] \big| [0\leadsto x],w(x)<w(v^{s\cup Q}_j) \right)\\
   & =\mathbf{P}\left([0\leadsto H_{i+1} -\lbrace e'(v_j^s),T^Q\rbrace] \big| \Gamma_2 \right)\\
   & =\mathbf{P}\left([0\leadsto H_{i+1} - e'(v_j^s)] \big| \Gamma_2 \right)
  \end{align*}

 \item  $\theta_1(H_{i+1})-\theta_1(H_{i})\geq 0$. 
 \begin{align*}
 &\theta_1(H_{i+1})-\theta_1(H_{i})= \mathbf{P} \left( [0\leadsto H_{i+1}] \right) - \mathbf{P} \left( [0\leadsto H_{i}] \right) \\
 &=\mathbf{P} \left( [0\leadsto H_{i+1}]\cap [0\leadsto H_{i+1} -e'(v)] \right)  -\mathbf{P} \left( [0\leadsto H_{i}]\cap [0\leadsto H_{i} -e(v)]  \right)\\
 &+\mathbf{P} \left( [0\leadsto H_{i+1}]\cap [0\not\leadsto H_{i+1} -e'(v)] \right)
 - \mathbf{P} \left( [0\leadsto H_{i}]\cap [0\not\leadsto H_{i} -e(v)]  \right)\\
 &=\mathbf{P} \left(  [0\leadsto H_{i+1} -e'(v)] \right)  
 -\mathbf{P} \left( [0\leadsto H_{i} -e(v)]  \right)\\
 &+\mathbf{P} \left( \Gamma_2\cap [0\not\leadsto H_{i+1} -e'(v)] \right)
 - \mathbf{P} \left( \Gamma_1\cap [0\not\leadsto H_{i} -e(v)]  \right)\\
 &=\mathbf{P} \left( \Gamma_2 \right) -\mathbf{P} \left( \Gamma_2\cap [0\leadsto H_{i+1} -e'(v)] \right) \\
 &- \mathbf{P} \left( \Gamma_1 \right) +\mathbf{P} \left( \Gamma_1\cap [0\leadsto H_{i} -e(v)]  \right)\\
 &=\mathbf{P} \left( \Gamma_1 \right)\left[ 
 \mathbf{P} \left( [0\leadsto H_{i} -e(v)] \big| \Gamma_1 \right) 
 -\mathbf{P} \left( [0\leadsto H_{i+1}-e'(v) ] \big| \Gamma_2 \right)  
 \right] \geq 0
 \end{align*}
 \end{itemize}
 
\end{proof}

\begin{lem}
	 $\theta_1(T')\geq \theta_1(H^k)$
\end{lem}

\begin{proof}
By definition we require to prove that: the probability of having a $1$-accessible path in $T'$, is an upper bound for the probability of having a $1$-accessible path in $H^k$. 
	
	Let $H'\subset H^k$ be the tree induced by the vertices of $H^k$ at distance at most $h/k$ to the root. Notice that $\theta_1(H') \geq \theta_1(H^k)$. Also notice that, if $v$ is a vertex of $H'$ at distance at most $h/k-1$ to the root then (for $h$ large) \[ \text{deg}_{H'}(v)=\sum_{j=1}^k \left\lfloor f(h) \right\rfloor ^j  \leq \sum_{j=1}^k f(h) ^j = \frac{f(h)^{k+1}-1}{f(h)-1}\]  
	We claim that $\text{deg}_{H'}(v)\leq g(\lfloor h/k \rfloor)$ (recall, $g$ was defined as $g(h)=\frac{h}{e}$), from which  $H'\subset T'$ and, as both trees have the same height, then $\theta_1(T') \geq \theta_1(H')$.
	
	Now, to finish this proof, we prove that the claim holds. Let $A=\sqrt[k]{h/(ek)}$. As $f(h)\leq  \sqrt[k]{h/(ek)}- \Omega(h^c)$ then for $c<x<1$ and $h$ large enough, $ f(h)\leq A-A^{x}$. It is enough to prove that 
	\[ \frac{\left(A-A^{x} \right)^{k+1} -1}{\left(A-A^{x} \right)-1}\leq A^k\]
	but it follows from that 
	\[\left(A-A^{x} \right)^{k+1} -1= A^{k+1}-(k+1)A^{k+x}+o(A^{k+x})\] and  
	\[ A^k\left(A-A^{x} -1 \right)= A^{k+1}-A^{k+x} +A^k.\]
\end{proof}

\noindent\textbf{Acknowledgments:} 
The authors are thankful to Ricardo Restrepo for helpful discussions.

%Thanks are due to the anonymous referees for their careful reading, criticism and suggestions which helped us to considerably improve the paper.

%%%%%%%%%%%%%%%%%%%%%%%%%%%%%%%%%%%%%%%%%%%%%%%%%%%%%%%%%%%%%%%%%%%%%%%%%%%
\bibliographystyle{plain} \bibliography{kpercolation}

\begin{thebibliography}{10}

\bibitem{ApHc2016}
J.~Berestycki, E.~Brunet, and Z.~Shi.
\newblock {. The number of accessible paths in the hypercube}.
\newblock {\em {Bernoulli}}, 22(2):653--680, 2016.

\bibitem{BBS2017}
Julien Berestycki, Éric Brunet, and Zhan Shi.
\newblock {Accessibility percolation with backsteps}.
\newblock {\em ALEA, Lat. Am. J. Probab. Math. Stat.}, 14:45-- 62, 2017.

\bibitem{JKAJ2011}
Jasper Franke, Alexander Klözer, J.~Arjan G.~M. de~Visser, and Joachim Krug.
\newblock Evolutionary accessibility of mutational pathways.
\newblock {\em PLOS Computational Biology}, 7(8):1--9, 08 2011.

\bibitem{G1984}
John~H. Gillespie.
\newblock Molecular evolution over the mutational landscape.
\newblock {\em Evolution}, 38(5):1116--1129, 1984.

\bibitem{Martinsson2014}
Peter Hegarty and Anders Martinsson.
\newblock On the existence of accessible paths in various models of fitness
  landscapes.
\newblock {\em Ann. Appl. Probab.}, 24(4):1375--1395, 2014.

\bibitem{Li2017}
Li~Li.
\newblock Phase transition for accessibility percolation on hypercubes.
\newblock {\em Journal of Theoretical Probability}, Jul 2017.

\bibitem{martinsson2015}
Anders Martinsson.
\newblock Accessibility percolation and first-passage site percolation on the
  unoriented binary hypercube.
\newblock {\em arXiv preprint arXiv:1501.02206}, 2015.

\bibitem{NK2013}
S.~Nowak and J.~Krug.
\newblock Accessibility percolation on n-trees.
\newblock {\em EPL (Europhysics Letters)}, 101(6):66004, 2013.

\bibitem{FitnessLandscape2014Book}
Hendrik Richter and Andries Engelbrecht.
\newblock {\em Recent advances in the theory and application of fitness
  landscapes}.
\newblock Springer, 2014.

\bibitem{RZ2013}
Matthew Roberts and Lee Zhao.
\newblock Increasing paths in regular trees.
\newblock {\em Electron. Commun. Probab.}, 18:no. 87, 1--10, 2013.

\bibitem{DefLandscapes2003}
Peter~F Stadler and Christopher~R Stephens.
\newblock Landscapes and effective fitness.
\newblock {\em Comments{\textregistered} on Theoretical Biology},
  8(4-5):389--431, 2003.

\bibitem{WC2005}
Daniel~M. Weinreich and Lin Chao.
\newblock Rapid evolutionary escape by large populations from local fitness
  peaks is likely in nature.
\newblock {\em Evolution}, 59(6):1175--1182, 2005.

\bibitem{WDDH2006}
Daniel~M. Weinreich, Nigel~F. Delaney, Mark~A. DePristo, and Daniel~L. Hartl.
\newblock Darwinian evolution can follow only very few mutational paths to
  fitter proteins.
\newblock {\em Science}, 312(5770):111--114, 2006.

\bibitem{WWC2005}
Daniel~M. Weinreich, Richard~A. Watson, and Lin Chao.
\newblock Perspective:sign epistasis and genetic constraint on evolutionary
  trajectories.
\newblock 2005.

\bibitem{WDFF2009}
DW~Weissman, MM~Desai, DS~Fisher, and MW~Feldman.
\newblock The rate at which asexual populations cross fitness valleys.
\newblock {\em Theoretical Population Biology}, 75:286--300, 2009.

\bibitem{LandscapeDef1932}
Sewall Wright.
\newblock The roles of mutation, inbreeding, crossbreeding, and selection in
  evolution.
\newblock In {\em Proceedings of the Sixth International Congress of Genetics},
  volume~1, pages 356--366, 1932.

\end{thebibliography}
%%%%%%%%%%%%%%%%%%%%%%%%%%%%%%%%%%%%%%%%%%%%%%%%%%%%%%%%%%%%%%%%%%%%%%%%%%%

\end{document}